\documentclass[12pt]{amsart}
\usepackage{amsfonts}
\usepackage{amsmath}
\usepackage{amssymb}
\usepackage[hmargin=1in, vmargin=1in]{geometry}

\newtheorem{theorem}{Theorem}[section]

\newtheorem{lemma}[theorem]{Lemma}

\theoremstyle{definition}
\newtheorem{definition}[theorem]{Definition}
\newtheorem{remark}[theorem]{Remark}
\newtheorem{example}[theorem]{Example}
\newtheorem{claim}{Claim}

\numberwithin{equation}{section}

\begin{document}
\title[Optimal Tauberian constant in Ingham's theorem]{Optimal Tauberian constant in Ingham's theorem for Laplace transforms}

\author[G. Debruyne]{Gregory Debruyne}
\thanks{G. Debruyne gratefully acknowledges support by Ghent University, through a BOF Ph.D. grant.}
\address{Department of Mathematics\\ Ghent University\\ Krijgslaan 281\\ 9000 Gent\\ Belgium}
\email{gregory.debruyne@UGent.be}

\author[J. Vindas]{Jasson Vindas}
\thanks{The work of J. Vindas was supported by the Research Foundation--Flanders, through the FWO-grant number 1520515N}
\address{Department of Mathematics\\ Ghent University\\ Krijgslaan 281\\ 9000 Gent\\ Belgium}
\email{jasson.vindas@UGent.be}
\subjclass[2010]{11M45, 40E20, 44A10}
\keywords{Ingham Tauberian theorem; Fatou-Riesz theorem; Tauberian constants;  Laplace transform}

\begin{abstract}
It is well known that there is an absolute constant $\mathfrak{C}>0$ such that if the Laplace transform $G(s)=\int_{0}^{\infty}\rho(x)e^{-s x}\:\mathrm{d}x$ of a bounded function $\rho$ has analytic continuation through every point of the segment $(-i\lambda ,i\lambda )$ of the imaginary axis, then 
$$
\limsup_{x\to\infty} \left|\int_{0}^{x}\rho(u)\:\mathrm{d}u - G(0)\right|\leq \frac{ \mathfrak{C}}{\lambda} \: \limsup_{x\to\infty} |\rho(x)|.
$$
The best known value of the constant $\mathfrak{C}$ was so far $\mathfrak{C}=2$. In this article we show that the inequality holds with $\mathfrak{C}=\pi/2$ and that this value is best possible. We also sharpen Tauberian constants in finite forms of other related complex Tauberian theorems for Laplace transforms.
\end{abstract}

\maketitle

\section{Introduction}

The aim of this article is to generalize and improve the following Tauberian theorem:
\begin{theorem}
\label{finiteformF-Rth1}
 Let $\rho\in L^{\infty}[0,\infty)$. Suppose that there is a constant $\lambda>0$ such that its Laplace transform 
\[
\label{F-Rthintroeq1} G(s)=\mathcal{L}\{\rho;s\}= \int_{0}^{\infty} \rho(x)e^{-s x}\:\mathrm{d}x 
\]
has analytic continuation through every point of the segment $(-i\lambda,i\lambda)$ of the imaginary axis and set $b=G(0)$. Then, there is an absolute constant $\mathfrak{C}>0$ such that  
\begin{equation}
\label{F-Rthintroeq1} \limsup_{x\to\infty}\left|\int_{0}^{x}\rho(u)\mathrm{d}u-b\right|\leq \frac{ \mathfrak{C}}{\lambda} \: \limsup_{x\to\infty} |\rho(x)| .
\end{equation}
\end{theorem}

Theorem \ref{finiteformF-Rth1} is a Laplace transform version of the Fatou-Riesz theorem \cite[Chap. III]{korevaarbook} and is due to Ingham \cite{ingham1935}, who established the inequality (\ref{F-Rthintroeq1}) with $\mathfrak{C}=6$. The absolute constant was improved to $\mathfrak{C}=2$ by Korevaar \cite{korevaar1982} and Zagier \cite{zagier1997} via Newman's contour integration method. Vector-valued variants of Theorem \ref{finiteformF-Rth1} have many important applications in operator theory, particularly in semigroup theory; see, for example, Arendt and Batty \cite{a-b}, Chill \cite{Chill1998}, and the book \cite{a-b-h-n}. We also refer to the recent works \cite{b-b-t2016,Chill2016} for remainder versions of Ingham's theorem. 

Recently \cite{Debruyne-VindasComplexTauberians}, the authors have weakened the assumption of analytic continuation on the Laplace transform in Theorem \ref{finiteformF-Rth1} to so-called local pseudofunction boundary behavior (see Section \ref{section preli} for the definition of this notion) of the analytic function (\ref{F-Rthintroeq2}), which includes as a particular instance $L^{1}_{loc}$-extension as well. The proof method given there could actually yield better values for $\mathfrak{C}$ than $2$, although sharpness could not be reached via that technique. 

Our goal here is to find the optimal value for $\mathfrak{C}$. So, the central part of this article is devoted to showing the ensuing sharp version of Theorem \ref{finiteformF-Rth1}. 

\begin{theorem}\label{finiteformF-Rth2} 
Let $\rho\in L^{\infty}[0,\infty)$. Suppose that there is a constant $b$ such that 
\begin{equation}
\label{F-Rthintroeq2} \frac{\mathcal{L}\{\rho;s\}-b}{s}
\end{equation}
admits local pseudofunction boundary behavior on the segment $(-i\lambda,i\lambda)$ of the imaginary axis. Then, the inequality \eqref{F-Rthintroeq1} holds with $\mathfrak{C}=\pi/2$. Moreover, the constant $\pi/2$ cannot be improved.
\end{theorem}

That $\mathfrak{C}=\pi/2$ is optimal in Theorem \ref{finiteformF-Rth2} will be proved in Subsection \ref{F-R subsection example} by finding an extremal example of a function for which the inequality (\ref{F-Rthintroeq1}) becomes equality. The Laplace transform of this example has actually analytic extension to the given imaginary segment, showing so the sharpness of $\pi/2$ under the stronger hypothesis of Theorem \ref{finiteformF-Rth1} as well. Our proof of Theorem \ref{finiteformF-Rth2} is considerably more involved than earlier treatments of the problem. It will be given in Section \ref{section proof of F-R theorem} and is based on studying the interaction of our extremal example with a certain extremal convolution kernel. 
 
The rest of the article derives several important consequences from Theorem \ref{finiteformF-Rth2}. We shall use Theorem \ref{finiteformF-Rth2} to sharpen Tauberian constants in finite forms of other complex Tauberian theorems. Section \ref{section generalizations and corollaries F-R} deals with generalizations and corollaries of Theorem \ref{finiteformF-Rth2} under two-sided Tauberian conditions, while we study corresponding problems with one-sided Tauberian hypotheses in Section \ref{section one-sided F-R}.  Our results can be regarded as general inequalities for functions whose Laplace transforms have local pseudofunction boundary behavior on a given symmetric segment of the imaginary axis in terms of their oscillation or decrease moduli at infinity. In particular, we shall show the ensuing one-sided version of Theorem \ref{finiteformF-Rth2} in Section \ref{section one-sided F-R}.

\begin{theorem}\label{finiteformF-R1sideth2} 
Let $\rho\in L^{1}_{loc}[0,\infty)$ be such that its Laplace transform is convergent on the half-plane $\Re e\:s>0$. If there is a constant $b$ such that \eqref{F-Rthintroeq2} has local pseudofunction boundary behavior on $(-i\lambda,i\lambda)$, then 
\begin{equation}\label{F-Rintroeq3}
\limsup_{x\to\infty}\left| \int_{0}^{x}\rho(u)\mathrm{d}u-b\right|\leq \frac{\mathfrak{K}}{\lambda} \max\left\{-\liminf_{x\to\infty}\rho(x),0\right\}
\end{equation}
with $\mathfrak{K}=\pi$. The constant $\pi$ here is best possible.
\end{theorem}

We point out that Ingham obtained a much weaker version of Theorem \ref{finiteformF-R1sideth2} in \cite[pp. 472--473]{ingham1935} with constant
\begin{equation}
\label{F-Rintroeq4}
\displaystyle\mathfrak{K}=8\left(\frac{\pi e^{2}}{2\int_{0}^{1/2} \sin^2x/x^{2}\:\mathrm{d}x}-1\right)\approx 182.91
\end{equation}
in the inequality \eqref{F-Rintroeq3}. We will deduce Theorem \ref{finiteformF-R1sideth2} from a corollary of Theorem \ref{finiteformF-Rth2} (cf. Theorem \ref{finiteformF-Rth5}) and the Graham-Vaaler sharp version of the Wiener-Ikehara theorem \cite{grahamvaaler}. We mention that Graham and Vaaler obtained the optimal constants in the finite form of the Wiener-Ikehara theorem via the analysis of extremal $L^{1}$-majorants and minorants for the exponential function.
\section{Preliminaries}
\label{section preli}
We briefly discuss in this section the notion of local pseudofunction boundary behavior. We refer to \cite{Debruyne-VindasW-I,Debruyne-VindasComplexTauberians,korevaar2005,korevaar2005FR} and \cite[Sect. III.14]{korevaarbook} for its connections with Tauberian theory. We normalize the Fourier transform as $\hat{\varphi}(t)=\mathcal{F}\{\varphi;t\}=\int_{-\infty}^{\infty}e^{-itx}\varphi(x)\:\mathrm{d}x$. As is standard, we interpret  Fourier transforms in the sense of tempered distributions when the integral definition does not make sense. 

The space of global pseudomeasures is $PM(\mathbb{R})=\mathcal{F}(L^{\infty}(\mathbb{R}))$. We call $f\in PM(\mathbb{R})$ a global pseudofunction if additionally $\lim_{|x|\to\infty} \hat{f}(x)=0$. A Schwartz distribution $g\in\mathcal{D}'(U)$ is said to be a local pseudofunction on an open set $U$ if every point of $U$ has a neighborhood where $g$ coincides with a global pseudofunction; we then write $g\in PF_{loc}(U)$.  Equivalently, the latter holds if and only if  $\lim_{|x|\to\infty} \widehat{\varphi g}(x)=0$ for every smooth compactly supported test function 
$\varphi\in \mathcal{D}(U)$. It should then be clear that smooth functions are multipliers for local pseudofunctions. Also, $L^{1}_{loc}(U)\subsetneq PF_{loc}(U)$, in view of the Riemann-Lebesgue lemma.

Let $G(s)$ be analytic on the half-plane $\Re e\:s>0$ and let $U\subset \mathbb{R}$ be open. We say that $G$ has local pseudofunction boundary behavior on the boundary open subset $iU$ of the imaginary axis if  $G$ admits a local pseudofunction as distributional boundary value on $i U$, that is, if there is $g\in PF_{loc}(U)$ such that 
\begin{equation*}
\lim_{\sigma\to 0^{+}}\int_{-\infty}^{\infty}G(\sigma+it)\varphi(t)\mathrm{d}t=\left\langle g(t),\varphi(t)\right\rangle\: , \quad \mbox{for each } \varphi\in\mathcal{D}(U).
\end{equation*}
We emphasize that $L^1_{loc}$-boundary behavior, continuous, or analytic extension are very special cases of local pseudofunction boundary behavior. 
 
Suppose that $G$ is given by the Laplace transform of a tempered distribution $\tau\in \mathcal{S}'(\mathbb{R})$ with $\operatorname*{supp} \tau \subseteq [0,\infty)$, i.e., $G(s)=\langle \tau(x), e^{-sx}\rangle$ for $\Re e\: s>0$. In this case the Fourier transform $\hat{\tau}$ is the boundary distribution of $G$ on the whole  imaginary axis $i\mathbb{R}$. Since $\widehat {\tau \ast \phi }= \hat{\tau} \cdot \hat{\phi}$, we conclude that $G$ admits local pseudofunction boundary behavior on $i U$ if and only if 
\begin{equation}
\label{convformlpbvb} \lim_{x\to\infty} (\tau \ast \phi)(x)=0, \quad \mbox{for all } \phi\in\mathcal{S}(\mathbb{R}) \mbox{ with } \hat{\phi}\in\mathcal{D}(U),
\end{equation}
where $\mathcal{S}(\mathbb{R})$ stands for the Schwartz class of rapidly decreasing smooth functions.

\section{Proof of Theorem \ref{finiteformF-Rth2}}\label{section proof of F-R theorem}

\subsection{A reduction}
We start by making some reductions. Define
$$
\tau(x)=\int_{0}^{x}\rho(u)\:\mathrm{d}u-b, \quad x\geq 0,
$$
and set $\tau(x)=0$ for $x<0$. The Laplace transform of $\tau$ is precisely the analytic function (\ref{F-Rthintroeq2}). We have to show that if $M>\limsup_{x\to\infty} |\rho(x)|$, then
\begin{equation}
\label{F-Rproofeq1}
\limsup_{x\to\infty} |\tau(x)|\leq \frac{M \pi}{2\lambda}.
\end{equation}
Denote as $\operatorname*{Lip}(I; C)$ the set of all Lipschitz continuous functions on a interval $I$ with Lipschitz constant $C$. We obtain that there is $X>0$ such that $\tau\in \operatorname*{Lip}([X,\infty); M)$. Since Laplace transforms of compactly supported functions are entire functions, the behavior of $\tau$ on a finite interval is totally irrelevant for the local pseudofunction behavior of its Laplace transform. It is now clear that Theorem \ref{finiteformF-Rth2} may be equivalently reformulated as follows, which is in fact the statement that will be shown in this section.

\begin{theorem}
\label{finiteformF-Rth3} Let $\tau\in L^{1}_{loc}[0,\infty)$ be such that $\tau\in\operatorname*{Lip}([X,\infty); M)$ for some sufficiently large $X>0$. If there is $\lambda>0$ such that $\mathcal{L}\{\tau;s\}$ admits local pseudofunction boundary behavior on $(-i\lambda, i\lambda)$, then \eqref{F-Rproofeq1} holds. The constant $\pi/2$ in this inequality is best possible.
\end{theorem} 
Next, we indicate that we may set w.l.o.g. $M = 1$ and $\lambda = 1$ in Theorem \ref{finiteformF-Rth3}. Indeed, suppose that we already showed the theorem in this case and assume that $\tau$ satisfies the hypotheses of the theorem for arbitrary $M$ and $\lambda$. Then $\tilde{\tau}(x) = M^{-1} \lambda \tau(\lambda^{-1}x)$ satisfies the hypotheses of Theorem \ref{finiteformF-Rth3} with $M = 1$ and $\lambda = 1$. Thus $\limsup_{x\to\infty} \left|\tilde{\tau}(x)\right| \leq \pi/2$, giving the desired result for $\tau$. Similarly if one finds some function showing that the result is sharp with $M=1$ and $\lambda=1$, the same transformation would lead to  the sharpness for arbitrary $M$ and $\lambda$.
\subsection{An example showing the optimality of the theorem}
\label{F-R subsection example}We now give an example for $\tau$ showing that Theorem \ref{finiteformF-Rth3} is sharp. The proof of the theorem itself will largely depend on this example. Define
\begin{equation*}
 \tau(x) = \begin{cases} 
0 & \mbox{if }  x \leq 0, \\
x, & \mbox{if } 0 \leq x \leq \pi/2,\\
-x + N\pi/2 & \mbox{if }  (N-1)\pi/2 \leq x \leq(N+1)\pi/2 \mbox{ for } N \equiv 2 \ (\mathrm{mod} \ 4), \\
x - N\pi/2 & \mbox{if } (N-1)\pi/2 \leq x \leq (N+1)\pi/2 \mbox{ for  } N \equiv 0 \ (\mathrm{mod} \ 4),
\end{cases}
\end{equation*}
where $N$ stands above for a positive integer.
Calculating its Laplace transform, one finds
\[
 \mathcal{L}\{\tau;s\} = \frac{1}{s^{2}} - \frac{2e^{-\pi s/2}}{s^{2}(1+e^{-\pi s})} = \frac{(1-e^{-\pi s/2})^{2}}{s^{2}(1+e^{-\pi s})}, \ \ \ \Re e \: s > 0,
\]
which admits analytic continuation through the segment $(-i, i)$, and thus has local pseudofunction boundary behavior on this interval of the imaginary axis. The function $\tau$ satisfies the hypotheses of Theorem \ref{finiteformF-Rth3} with $M = 1$ and $\lambda = 1$; we have here $\limsup_{x\to\infty} \left|\tau(x)\right| = \pi/2$. Hence the constant $\pi/2$ in \eqref{F-Rproofeq1} cannot be improved. If one wants an example for the sharpness of Theorem \ref{finiteformF-Rth2} (and Theorem \ref{finiteformF-Rth1}), one may take the piecewise constant function $\rho=\tau'$ as such an example.
\subsection{Analysis of a certain extremal function} The proof of the theorem will also depend on the properties of a certain extremal test function, namely,
\begin{equation*}
 K(x) = \frac{2\cos x}{\pi^{2}-4x^{2}}.
\end{equation*}
This function has many remarkable properties in connection to several extremal problems \cite{logan1983} and has already shown useful in Tauberian theory \cite{ingham1936,korevaarbook}.
Let us collect some facts that are relevant for the proof of Theorem \ref{finiteformF-Rth3}. Its Fourier transform is
\[
 \hat{K}(t) = \begin{cases}
 \cos(\pi t/2) & \mbox{if } \left|t\right|\leq 1,\\
 0 & \mbox{if } \left|t\right|\geq 1.
\end{cases}
\]
It satisfies \cite[Chap. III, Prop. 11.2]{korevaarbook} 
\begin{equation} \label{eqtestintegral}
 2 \int^{\pi/2}_{-\pi/2} K(x) \mathrm{d}x = \int_{-\infty}^{\infty}  \left|K(x)\right| \mathrm{d}x. 
\end{equation}
More important however, we need to know how this test function interacts with a modified version of our supposed extremal example, which we will denote throughout the rest of this section as
\begin{equation*} \alpha(x) := \begin{cases}
(N+1)\pi/2 - \left|x\right| & \mbox{if } N\pi/2 \leq \left|x\right| \leq (N+2)\pi/2 \mbox{ for  }N \equiv 0 \:  (\mbox{mod } 4),\\
-(N+1)\pi/2 + \left|x\right| & \mbox{if } N\pi/2 \leq \left|x\right| \leq (N+2)\pi/2 \mbox{ for } N \equiv 2 \: (\mbox{mod }  4).\\
\end{cases}
\end{equation*}
\begin{lemma} \label{lemtestexample} We have
\[
2\int^{\pi/2}_{-\pi/2} K(x) \alpha(x)\: \mathrm{d}x = \int_{-\infty}^{\infty} \left|K(x) \alpha(x)\right|\mathrm{d}x \quad \mbox{ and } \quad \int_{-\infty}^{\infty}K(x) \alpha(x) \mathrm{d}x = 0.
\]
\end{lemma}
\begin{proof}
Indeed, realizing that $K(x)\alpha(x)$ is positive when $\left|x\right| < \pi/2$ and negative otherwise, it suffices to show that the integral of $K(x) \alpha(x)$ on $(-\infty,\infty)$ is $0$, or equivalently on $(0,\infty) $ since $K(x)\alpha(x)$ is even. We split the integral in intervals of length $\pi/2$. Let $N\in\mathbb{N}$ be divisible by $4$. Then,
\begin{align*}
\int^{(N+1)\pi/2}_{N\pi/2} K(x)\alpha(x)\mathrm{d}x & = \frac{1}{2\pi}\int^{(N+1)\pi/2}_{N\pi/2} \frac{(\pi(N+1) -2x) \cos x}{\pi + 2x} 
+ \frac{(\pi(N+1) -2x) \cos x}{\pi - 2x}\: \mathrm{d}x\\
& = \frac{1}{2\pi}\int^{\pi/2}_{0} \frac{(\pi - 2x)\cos x}{(N+1)\pi + 2x} - \frac{(\pi - 2x)\cos x}{(N-1)\pi + 2x}\: \mathrm{d}x,
\end{align*}
and
\begin{align*}
\int^{(N+2)\pi/2}_{(N+1)\pi/2} K(x)\alpha(x)\mathrm{d}x & = \frac{1}{2\pi}\int^{(N+2)\pi/2}_{(N+1)\pi/2} \frac{(\pi(N+1) -2x) \cos x}{\pi + 2x}  
+ \frac{(\pi(N+1) -2x) \cos x}{\pi - 2x}\: \mathrm{d}x
\\
& = \frac{1}{2\pi}\int^{0}_{-\pi/2} \frac{-(-\pi - 2x)\cos x}{(N+3)\pi + 2x} + \frac{-(-\pi - 2x)\cos x}{(-N-1)\pi - 2x}\:  \mathrm{d}x\\
& = \frac{1}{2\pi}\int^{\pi/2}_{0} \frac{(\pi - 2x)\cos x}{(N+3)\pi - 2x} - \frac{(\pi - 2x)\cos x}{(N+1)\pi - 2x}\:  \mathrm{d}x.
\end{align*}
Similarly,
\begin{equation*}
 \int^{(N+3)\pi/2}_{(N+2)\pi/2} K(x)\alpha(x)\mathrm{d}x = \frac{1}{2\pi}\int^{\pi/2}_{0} \frac{(\pi - 2x)\cos x}{(N+3)\pi + 2x} - \frac{(\pi - 2x)\cos x}{(N+1)\pi + 2x} \mathrm{d}x
\end{equation*}
and
\begin{equation*}
 \int^{(N+4)\pi/2}_{(N+3)\pi/2} K(x)\alpha(x)\mathrm{d}x = \frac{1}{2\pi}\int^{\pi/2}_{0} \frac{(\pi - 2x)\cos x}{(N+5)\pi - 2x} - \frac{(\pi - 2x)\cos x}{(N+3)\pi - 2x} \mathrm{d}x.
\end{equation*}
Summing over all $4$ pieces and over all $N\equiv 0 \:  (\mbox{mod } 4)$, we see that the sum telescopes and that sum of the remaining terms for $N = 0$ add up to $0$.  
\end{proof}
\begin{lemma} \label{lemdectest} $K'/K$ and $K$ are decreasing on $(0,\pi/2)$.
\end{lemma}
\begin{proof} We need to show that $(K'/K)'$ is negative on $(0,\pi/2)$, or, which amounts to the same, that $(\log K)''$ is negative there. This is equivalent to showing that $\log K $ is concave. It is thus sufficient to verify that $K$ is concave on $(0,\pi/2)$. We have for, $x \in (0,\pi/2)$,
\begin{equation*}
 K''(x) = -\frac{1}{2\pi}\int^{\infty}_{-\infty} e^{ixt} t^{2} \hat{K}(t)\mathrm{d}t = -\frac{1}{2\pi}\int^{1}_{-1} t^{2}\cos(xt)\cos(\pi t/2) \mathrm{d}t < 0.
\end{equation*}
The last calculation for $K''$ can easily be adapted to find that $K'(x) < 0$ for $x \in (0,\pi/2)$, hence $K$ is decreasing there.
\end{proof}
The next lemma can be shown by a simple computation.
\begin{lemma} \label{lemtranslatestest} Let $N\in\mathbb{N}$. The function $K(x+ N\pi)/K(x)$ reaches an extremum at $e_{N} := \pi(-N+\sqrt{N^{2}-1})/2$ and is monotone on $(-\pi/2,e_{N})$ and $(e_{N},\pi/2)$.
\end{lemma}

\subsection{Some auxiliary lemmas} We will also employ the following lemmas. The proof of the next lemma is simple and we thus omit it.

\begin{lemma} \label{lemmonotone} Let $\mu$ be a positive measure on $[a,b]$ and let $\phi$ be non-increasing. Let $f$ and $g$ be functions such that 
\[
	 \int_{[a,b]} f (x)\mathrm{d}\mu(x) = \int_{[a,b]} g(x) \mathrm{d}\mu (x)
\]
and there exists $c \in [a,b]$ such that $f(x) \geq g(x)$ on $[a,c]$ and $f(x) \leq g(x)$ on $(c,b]$. Then,
\[
 \int_{[a,b]} f(x)\phi(x) \mathrm{d}\mu(x) \geq \int_{[a,b]} g(x) \phi (x)\mathrm{d}\mu(x).
\]
\end{lemma}
Naturally, the above lemma can be adapted to treat negative measures $\mu$ and non-decreasing functions $\phi$ and we will also refer to these adaptations as Lemma \ref{lemmonotone}.

We will use the ensuing class of functions for estimations.

\begin{definition} We say that a function $z$ is an upper pointed zig-zag function on $[-\pi/2,\pi/2]$ if there is a $c \in [-\pi/2, \pi/2]$ such that $z$ can be written as
\[
 z(x) = \begin{cases}
(x-c) + z(c) \quad \text{if } x \in [-\pi/2,c],\\
-(x-c) + z(c) \quad \text{if } x \in [c,\pi/2]. 
\end{cases}
\]
A function $z$ is called lower pointed zig-zag if $-z$ is upper pointed zig-zag.
\end{definition}                                                             
The following lemma will be a key ingredient in our arguments. It will allow us to work with piecewise linear functions instead of the more general Lipschitz continuous functions with Lipschitz constant 1.
\begin{lemma} \label{lemintegralfar} Let $I$ and $s$ be constants such that
\begin{equation} \label{eqconditionIs}
\int^{\pi/2}_{-\pi/2} (s - (x+\pi/2))K(x) \mathrm{d}x \leq I \leq \int^{\pi/2}_{-\pi/2} (s + (x+\pi/2))K(x) \mathrm{d}x
\end{equation} 
and set
\[
 A = \left\{f\in \operatorname{Lip}([-\pi/2,\pi/2];1) \mid \: f(-\pi/2) = s,\:  \int^{\pi/2}_{-\pi/2} f(x) K(x) \mathrm{d}x \leq I \right\} 
\]
and 
\[
 B = \left\{z \mid z \text{ is upper pointed zig-zag, } z(-\pi/2) \leq s,\: \text{and } \int^{\pi/2}_{-\pi/2} z(x) K(x) \mathrm{d}x = I\right\} .
\]
Then,        
\begin{equation} \label{eqresultcrucial}
 \inf_{f \in A} \int^{\pi/2}_{-\pi/2} f(x) K(x+N\pi) \mathrm{d}x \geq \inf_{z \in B} \int^{\pi/2}_{-\pi/2} z(x) K(x+N\pi) \mathrm{d}x,
\end{equation}
if $K(x+N\pi)$ is negative on $(-\pi/2, \pi/2)$, i.e., when $N \geq 2$ is even.  
\end{lemma}
\begin{proof}
We set $C_{f} := \int^{\pi/2}_{-\pi/2} f(x) K(x+N\pi) \mathrm{d}x$. We may clearly assume that the inequality regarding the integral in the definition of $A$ is actually an equality. By Lemma \ref{lemtranslatestest}, there is $e_{N}$ such that $K(x+N\pi)/K(x)$ is non-increasing on $[-\pi/2,e_{N}]$ and non-decreasing on $[e_{N}, \pi/2]$. Let $f \in A$ arbitrary. We will construct a zig-zag function $z \in B$ for which $C_{f} \geq C_{z}$. Let us first consider $j$ which is defined on $[-\pi/2,e_{N}]$ as the straight line with slope $1$ such that $\int_{-\pi/2}^{e_{N}} f(x) K(x) \mathrm{d}x = \int_{-\pi/2}^{e_{N}} j(x) K(x) \mathrm{d}x$ and on $(e_{N},\pi/2]$ as the straight line with slope $-1$ such that $\int_{e_{N}}^{\pi/2} f(x) K(x) \mathrm{d}x = \int_{e_{N}}^{\pi/2} j(x) K(x) \mathrm{d}x$. The fact that $f\in \operatorname{Lip}([-\pi/2,\pi/2];1)$ allows us to apply Lemma \ref{lemmonotone} to the positive measure $K(x)\mathrm{d}x$ and the non-increasing (resp. non-decreasing) function $K(x+N\pi)/K(x)$ on the interval $[-\pi/2, e_{N}]$ (resp. $[e_{N}, \pi/2]$) to find that $\int f(x) K(x+ N\pi)\mathrm{d}x \geq \int j(x) K(x+ N\pi)\mathrm{d}x$ on both intervals $[-\pi/2, e_{N}]$ and $[e_{N}, \pi/2]$. Note that we must necessarily have $j(-\pi/2)\leq f(-\pi/2)=s$, by the Lipschitz continuity of $f$. The function $j$ may not be a zig-zag function however, as $j$ may have a discontinuity at $e_{N}$. We set $z = j$ on the interval where $j$ takes the lowest value at $e_{N}$, i.e., if $j(e_{N}) \leq \lim_{x \rightarrow e_{N}^{+}} j(x)$ then we set $z := j$ on $[-\pi/2, e_{N}]$, otherwise we set $z := j$ on $[e_{N}, \pi/2]$. Notice that in the first case we have by construction that $z(-\pi/2) \leq s$ and (in both cases) $z(e_{N}) \geq f(e_{N})$. We then extend $z$ on the remaining interval as the unique upper pointed zig-zag function such that $\int f(x) K(x) \mathrm{d}x = \int z(x) K(x) \mathrm{d}x$ there and makes $z$ a continuous function at the point $e_{N}$. As was the case for the case for the comparison between $j$ and $f$, one can use Lemma \ref{lemmonotone} (in exactly the same way) to compare $f$ and $z$ and conclude that $\int f(x) K(x+ N\pi)\mathrm{d}x \geq \int z(x) K(x+ N\pi)\mathrm{d}x$ on both intervals $[-\pi/2,e_{N}]$ and $[e_{N},\pi/2]$; whence $C_{f} \geq C_{z}$. From the construction it is also clear that $z(-\pi/2) \leq s$ and thus $z \in B$. 
\end{proof}

Lemma \ref{lemintegralfar} has an obvious analogue when $K(x+N\pi)$ is positive on the interval $(-\pi/2, \pi/2)$, namely, when $N \geq 1$ is odd. One then needs to replace in the definition of $B$ upper pointed zig-zag functions by lower pointed ones and the inequality $z(-\pi/2) \leq s$ needs to be reversed, and in the definition of $A$ the inequality regarding the integral also has to be reversed. The proof is basically the same and will therefore be omitted. This analogue will also be referred to as Lemma \ref{lemintegralfar}. We also note that it is easy to see that the infimum in (\ref{eqresultcrucial}) with respect to the set $B$ is in fact a minimum.

\subsection{The actual proof}\label{subsection F-Rproof}
We now come to the proof of Theorem \ref{finiteformF-Rth3}. We may modify $\tau$ on $[0,X]$ in any way we like because this does not affect the local pseudofunction behavior of its Laplace transform. So, we assume that $\tau\in \operatorname*{Lip}(\mathbb{R};1)$, namely,
\begin{equation}
\label{Lip1eq}
|\tau(x)-\tau(y)|\leq |x-y|, \quad \forall x,y\in \mathbb{R}.
\end{equation}  

The Lipschitz continuity of $\tau$ gives the bound $\tau(x)=O(x)$, so that we can view $\tau$ as a tempered distribution with support in $[0,\infty)$. As indicated in Section \ref{section preli}, the local pseudofunction boundary behavior of the Laplace transform  $\mathcal{L}\{\tau;s\}$ on $(-i,i)$ then yields (\ref{convformlpbvb}) with $U=(-1,1)$. From here we can prove that $\tau$ is bounded\footnote{This also follows directly from \cite[Thm.  3.1]{Debruyne-VindasComplexTauberians}, where we have shown that a much weaker one-sided Tauberian condition (bounded decrease) suffices to deduce boundedness. In the case under consideration we have however a two-sided condition and the proof of the assertion then becomes much easier and shorter.}. In fact, select a non-negative test function $\phi\in\mathcal{S}(\mathbb{R})$ with Fourier transform supported in $(-1,1)$ and $\int_{-\infty}^{\infty}\phi(x)\mathrm{d}x=1$. Applying \eqref{convformlpbvb} and \eqref{Lip1eq}, we obtain\footnote{One obtains the bound $\limsup_{x\to\infty}|\tau(x)|\leq\int_{-\infty}^{\infty}\left|x\phi(x)\right|\mathrm{d}x$, which, upon a density argument, remains valid for all $\phi\in L^{1}(\mathbb{R}, (|x|+1)\mathrm{d}x)$ such that $\int_{-\infty}^{\infty}\phi(x)\mathrm{d}x=1$ and $\operatorname*{supp}\hat{\phi} \subseteq [-1,1]$. Ingham obtained \cite{ingham1935} $\limsup_{x\to\infty}|\tau(x)|< 6$ by choosing the Jackson kernel $\displaystyle\phi(x)=96\sin^{4} (x/4)/(\pi x^{4})$ and using an intermediate inequality (cf. \cite[Lemma (II), p. 465]{ingham1935}) for $\int_{-\infty}^{\infty}\left|x\phi(x)\right|\mathrm{d}x$. Explicit evaluation of the latter integral (cf. \cite[Eq. (ii.b), p. 448]{Kolk2003}) however delivers the much better inequality $\limsup_{x\to\infty}|\tau(x)|\leq (12/\pi)\int_{0}^{\infty}\sin^{4}x/x^3\mathrm{d}x= (12\log 2)/\pi\approx 2.65$.}
$$
|\tau(h)|= \left|\int_{-\infty}^{\infty}(\tau(h)-\tau(x+h))\phi(x)\mathrm{d}x+o(1)\right|\leq \int_{-\infty}^{\infty}\left|x\phi(x)\right|\mathrm{d}x+o(1)=O(1).
$$
Since we now know that $\tau\in L^{\infty}(\mathbb{R})$, we conclude that $(\tau\ast \phi)(h)=o(1)$ actually holds for all $\phi$ in the closure of $\mathcal{F}(\mathcal{D}(-1,1))$, taken in the Banach space $L^{1}(\mathbb{R})$, i.e., for every $L^{1}$-function $\phi$ whose Fourier transform vanishes outside $[-1,1]$. This means that we can take here the extremal kernel $\phi=K$. Summarizing, we have arrived to the key relation
\begin{equation}
\label{convolutioneqK}
\lim_{h\to\infty} \int_{-\infty}^{\infty} \tau(x+h)K(x)\mathrm{d}x=0.
\end{equation}
In the sequel, we only make use of (\ref{Lip1eq}) and (\ref{convolutioneqK}).

The idea of the proof of the inequality
\begin{equation}
\label{F-Rproofeq1bis}
\limsup_{x\to\infty} |\tau(x)|\leq \frac{\pi}{2}
\end{equation}
goes as follows. If $\tau(h)$ is `too large', then the Lipschitz condition (\ref{Lip1eq}) forces that a substantial portion of the integral $\int^{\infty}_{-\infty} \tau(x+h)K(x) \mathrm{d}x$ comes from the contribution of a neighborhood of the origin. If this is too excessive ($\tau(h)$ is too large), the tails of the integral will not be able to compensate this excess and the total integral will be large, violating the condition (\ref{convolutioneqK}). 

Let $\varepsilon>0$ be a small number that will be specified later. Our analysis makes use of the smooth function
\[
 f(y) := \int^{\pi/2}_{-\pi/2} \tau(x+y)K(x)\mathrm{d}x.
\]
The function $\tau$ is bounded, so is $f$. We may modify $\tau$ on a finite interval in such a fashion that \eqref{Lip1eq} still holds and the global supremum of $|f|$ stays sufficiently close to its limit superior at infinity. Furthermore, changing $\tau$ in this way  does not affect our hypothesis \eqref{convolutioneqK}. Thus, we assume 
\begin{equation}
\label{F-Rsup-limsupcond}
\limsup_{y\to\infty}\left| \int^{\pi/2}_{-\pi/2} \tau(x+y)K(x)\mathrm{d}x\right|\geq \sup_{y\in\mathbb{R}}\left| \int^{\pi/2}_{-\pi/2} \tau(x+y)K(x)\mathrm{d}x\right|-\frac{\varepsilon}{2}.
\end{equation}

Let us choose $h>0$ such that
\begin{equation} \label{eqconvolutionrelation}
 \left|\int^{\infty}_{-\infty} \tau(x+h) K(x) \mathrm{d}x\right| \leq \varepsilon
\end{equation}
and
\[
 \left|\int^{\pi/2}_{-\pi/2} \tau(x+h) K(x) \mathrm{d}x\right|
\]
is `maximal', i.e. (assuming w.l.o.g. that it is positive),
\begin{equation} \label{eqmaxcond1}
 \int^{\pi/2}_{-\pi/2} \tau(x+h) K(x) \mathrm{d}x >  
 \sup_{y\in\mathbb{R}}\left|\int^{\pi/2}_{-\pi/2} \tau(x+y) K(x) \mathrm{d}x\right| - \varepsilon
\end{equation}
and
\begin{equation} \label{eqmaxcond2}
 0 \leq f'(h) <\frac{ \varepsilon}{2\pi} = \varepsilon K(\pi/2).
\end{equation}
That such a choice of $h$ is possible simply follows from the fact that $f$ is bounded and \eqref{F-Rsup-limsupcond}. Indeed, assuming without loss of generality that the set of all $h$ such that \eqref{eqmaxcond1} holds is infinite (the condition \eqref{F-Rsup-limsupcond} ensures that it is non-empty and unbounded), we have that either $f$ has infinitely many local maxima accumulating to $\infty$ on this set or that there is an neighborhood of $\infty$ where $f$ is increasing. In the latter case $f$ would have a limit and $\liminf_{y\to\infty}f'(y)=0$.
 
Let us now suppose that \eqref{F-Rproofeq1bis} does not hold, that is, 
\begin{equation}\label{finiteFR-contradictioneq}
\limsup_{x\to\infty} |\tau(x)|> \eta+\pi/2.
\end{equation} 
for some $\eta>0$. Our task in the rest of the section is to prove that \eqref{finiteFR-contradictioneq} conflicts with \eqref{eqconvolutionrelation}.

 We choose $\beta_{0}$ and $\beta_{1}$ in such a way that
\begin{equation} \label{eqdefbeta0}
 \int^{\pi/2}_{-\pi/2} \tau(x+h) K(x) \mathrm{d}x = \int^{\pi/2}_{-\pi/2} (\beta_{0} + \alpha(x))K(x)\mathrm{d}x
\end{equation}
and
\begin{equation} \label{eqdefbeta1}
\sup_{y\in\mathbb{R}}\left|\int^{\pi/2}_{-\pi/2} \tau(x+y) K(x) \mathrm{d}x\right| = \int^{\pi/2}_{-\pi/2} (\beta_{1} + \alpha(x))K(x)\mathrm{d}x.
\end{equation}
From (\ref{eqmaxcond1}), it follows that $\beta_{0} > \beta_{1} - c\varepsilon$, where $c = (\int^{\pi/2}_{-\pi/2}K(x)\mathrm{d}x)^{-1}>0$. We also have
$\beta_{1}> 0$, as follows from \eqref{finiteFR-contradictioneq} and \eqref{Lip1eq}. We actually have the lower bound\footnote{Here $\beta_1$ depends on $\varepsilon$ because of our assumption (\ref{F-Rsup-limsupcond}), but, in contrast, the constant $\eta$ is independent of it. The lower bound (\ref{FR-lowerboundbeta1}) then plays a role below.}
\begin{equation}
\label{FR-lowerboundbeta1}
\beta_1>  \eta.
\end{equation} 
 In fact, if $y$ is a point where $|\tau(y)|>  \eta'+\pi/2 $ with $\eta'>\eta$, the Lipschitz condition \eqref{Lip1eq} implies that $|\tau(x+y)|> \eta'+\alpha(x)$ and $\tau(x+y)$ also has the same sign as $\tau(y)$ for all $x\in [-\pi/2,\pi/2]$; hence, \eqref{eqdefbeta1} yields \eqref{FR-lowerboundbeta1}.
\begin{claim} \label{claim} Let $h$ be chosen as above, then $\tau(h+\pi/2) \geq \beta_{0}-\varepsilon$ as well as $\tau(h-\pi/2) \geq \beta_{0}-\varepsilon$.
\end{claim}
Indeed, the Lipschitz condition \eqref{Lip1eq} implies that they cannot be both smaller than $\beta_{0}$, as (\ref{eqdefbeta0}) could otherwise not be realized. Suppose w.l.o.g that $\tau(h-\pi/2) < \beta_{0} - \varepsilon$ and $\tau(h+\pi/2)\geq \beta_{0}$. We will show that this violates the maximality assumption (\ref{eqmaxcond2}). We have
\[
 f'(h) = -\int^{\pi/2}_{-\pi/2} \tau(x+h)K'(x)\mathrm{d}x + \tau(\pi/2 + h)K(\pi/2) - \tau(-\pi/2+h)K(-\pi/2).
\]
To prove the claim, it thus suffices to show that $\int^{\pi/2}_{-\pi/2} K'(x) \tau(x+h)\mathrm{d}x\leq 0$. Noting that $\int^{\pi/2}_{-\pi/2} K'(x) (\beta_{0} + \alpha(x))\mathrm{d}x = 0$ and subtracting $\beta_{0} + \alpha(x)$ from $\tau(x+h)$, it would be sufficient to prove that there is no function $\rho(x)$ such that $\rho(-\pi/2) < 0$, $\rho$ is non-increasing on $[-\pi/2,0]$, non-decreasing on $[0,\pi/2]$, $\int^{\pi/2}_{-\pi/2} K'(x) \rho(x)\mathrm{d}x > 0$ and $\int^{\pi/2}_{-\pi/2} K(x) \rho(x)\mathrm{d}x = 0$. Suppose that there is such a function $\rho$. Since $\rho(-\pi/2) < 0$, $\rho$ is non-increasing on $[-\pi/2,0]$ and $K'$ is positive on $(-\pi/2,0)$, it follows that $\int^{0}_{-\pi/2} \rho(x) K'(x) \mathrm{d}x < 0$. We set $R$ as the constant such that $\int^{\pi/2}_{0} \rho(x) K(x) \mathrm{d}x = R\int^{\pi/2}_{0} K(x) \mathrm{d}x$. Note that $R \geq 0$ since $\int^{\pi/2}_{-\pi/2} \rho(x) K(x) \mathrm{d}x = 0$
 and $\int^{0}_{-\pi/2} \rho(x) K(x) \mathrm{d}x < 0$, because $\rho$ is negative on $(-\pi/2,0)$.
We apply Lemma \ref{lemmonotone} with the positive measure $K(x) \mathrm{d}x$ and the weight function $\phi(x) = K'(x)/K(x)$ in order to compare the function $\rho$ with the constant function $R$ on the interval $[0,\pi/2]$. By Lemma \ref{lemdectest} the function $K'(x)/K(x)$ is non-increasing and by the non-decreasing property of $\rho$, we obtain
\[
	\int^{\pi/2}_{0} K'(x)\rho(x) \mathrm{d}x \leq \int^{\pi/2}_{0}RK'(x) \mathrm{d}x \leq 0,
\]
since $K'$ is negative on $(0,\pi/2)$. We thus obtain $\int^{\pi/2}_{-\pi/2}K'(x)\rho(x) \mathrm{d}x \leq 0 $, violating one of the properties $\rho$ needed to satisfy. Hence $\rho$ cannot exist and the proof of the claim is complete.

\bigskip

Let us now define the auxiliary function $\gamma$:

\begin{equation*} 
 \gamma(x) := \begin{cases}
\beta_{0} + \alpha(x) & \mbox{if } \left|x\right| < \pi/2,\\
\beta_{1}/2 + \alpha(x) & \mbox{if } \pi/2 \leq \left|x\right| < \pi,\\
\beta_{2} + \alpha(x) & \mbox{if } \pi \leq \left|x\right| \leq 3\pi/2,\\
\beta_{1} + \alpha(x) & \mbox{in the other cases when } \alpha(x) \geq 0,\\
-\beta_{1} + \alpha(x) & \mbox{in the other cases when } \alpha(x) < 0,
\end{cases}
\end{equation*}
where $\beta_{2}$ is chosen in such a way that 
\begin{equation}\label{FR-defbeta2}
\int^{\pi/2}_{-\pi/2} \gamma(x+\pi)K(x)\mathrm{d}x = -\int^{\pi/2}_{-\pi/2} (\beta_{1}+ \alpha(x))K(x)\mathrm{d}x.
\end{equation}
(Note that $\beta_{2}=-5\beta_1/2 < - \beta_{1}$.) 

 We intend to show that
 \begin{equation} 
\label{eqintegralinequality}
\int_{-\infty}^{\infty} \tau(x+h)K(x)\mathrm{d}x \geq \int_{-\infty}^{\infty} \gamma(x)K(x)\mathrm{d}x > \varepsilon.
\end{equation}
This would conclude the proof as (\ref{eqconvolutionrelation}) is violated and hence $\limsup_{x\to\infty} |\tau(x)| \leq \pi/2$ must hold.

We first prove that $\int^{\infty}_{-\infty} \gamma(x)K(x)\mathrm{d}x > \varepsilon$. Let 
\[ 
 \tilde{\gamma}(x) := \begin{cases}
\beta_{1} + \alpha(x) & \mbox{if } \alpha(x) \geq 0,\\
-\beta_{1} + \alpha(x) & \mbox{if } \alpha(x) < 0.
\end{cases}
\]
It is clear that $\int^{\infty}_{-\infty} \tilde{\gamma}(x) K(x)\mathrm{d}x = 0$ due to Lemma \ref{lemtestexample} and (\ref{eqtestintegral}). A small computation gives
\begin{align*}
\int^{3\pi/2}_{\pi/2} (\gamma(x) - \tilde{\gamma}(x))K(x)\mathrm{d}x&= 48\pi\beta_1\int_{0}^{\pi/2} \frac{x\cos x}{(\pi^2-4x^{2})(9\pi^2-4x^2)}\:\mathrm{d}x
\\
&>b= 48\pi\eta\int_{0}^{\pi/2} \frac{x\cos x}{(\pi^2-4x^{2})(9\pi^2-4x^2)}\:\mathrm{d}x,
\end{align*}
where we have used \eqref{FR-lowerboundbeta1}. All involved functions are even, so $\int_{-3\pi/2}^{-\pi/2} (\gamma(x) - \tilde{\gamma}(x))K(x)\mathrm{d}x>b$. 
The only other contribution for $\int^{\infty}_{-\infty} (\gamma(x) - \tilde{\gamma}(x))K(x)\mathrm{d}x$ comes from the interval $[-\pi/2,\pi/2]$ and it is precisely $(\beta_0-\beta_1)\int_{-\pi/2}^{\pi/2}K(x)\mathrm{d}x>-\varepsilon$.  We obtain $\int_{-\infty}^{\infty}\gamma(x)K(x)\mathrm{d}x> 2b-\varepsilon,$
 which gives the second inequality in (\ref{eqintegralinequality}) if we choose $\varepsilon<b$, as we may do.

The proof will be complete if we show the first inequality of (\ref{eqintegralinequality}).
It is clear that the inequality 
 \begin{equation} 
\label{eqintegralinequalityI}
\int_{J} \tau(x+h)K(x)\mathrm{d}x \geq \int_{J} \gamma(x)K(x)\mathrm{d}x
\end{equation}
 holds (as an equality, cf. \eqref{eqdefbeta0}) if we restrict the domain of integration to $J=[-\pi/2,\pi/2]$. We will extend the domain of the integration in \eqref{eqintegralinequalityI} to $J=[-\pi/2, N\pi/2]$ for all $N$. The arguments we will give will be symmetrical (see also Claim \ref{claim}) and it can be readily seen that they work to get the inequality \eqref{eqintegralinequalityI} on all intervals of the form $J=[-N\pi/2,N\pi/2]$. Thus, since $N$ can be chosen arbitrarily large, it then suffices to prove \eqref{eqintegralinequalityI} if the intervals of integration are $J=[-\pi/2, N\pi/2]$. 

By Claim \ref{claim}, we have that 
\begin{equation}
\label{FReqchoiceepsilon}
\tau(h + \pi/2) > \beta_{0} - \varepsilon > \beta_{1}/2
\end{equation}
 if $\varepsilon$ is small enough. In fact, using (\ref{FR-lowerboundbeta1}), the choice $\varepsilon<\eta/(2c+2)$ suffices. By the Lipschitz condition \eqref{Lip1eq} and \eqref{FReqchoiceepsilon}, we obtain that $\tau(x+h)>\gamma(x)$ on the interval $[\pi/2,\pi]$, and, combining this with the fact that $K$ is positive on this interval, we see that \eqref{eqintegralinequalityI} also holds on $J=[\pi/2,\pi]$ and hence on the interval $J=[-\pi/2,\pi]$.

For the next interval we apply Lemma \ref{lemintegralfar} with $I=-\int^{\pi/2}_{-\pi/2} (\beta_{1} + \alpha(x)) K(x) \mathrm{d}x$ and $s=\tau(h+\pi/2)$. Notice that $I\leq \int_{-\pi/2}^{\pi/2}\tau(x+h+\pi)\mathrm{d}x$, due to (\ref{eqdefbeta1}). (It could still happen that $s$ is so large that the hypothesis for the lower bound (\ref{eqconditionIs}) for $I$ in Lemma  \ref{lemintegralfar} is not fulfilled. If this is the case we pick for $z$ the lower pointed zig-zag function with $z(\pi/2) = s$ and slope $-1$ on the interval $[\pi/2,3\pi/2]$. The proof then goes along similar lines with only mild adjustments.) We obtain a lower pointed zig-zag function $z(x)$ on $[\pi/2,3\pi/2]$ with starting point 
$z(\pi/2) \geq \tau(h+\pi/2)$ such that $\int^{\pi/2}_{-\pi/2} z(x+\pi) K(x) \mathrm{d}x = -\int^{\pi/2}_{-\pi/2} (\beta_{1} + \alpha(x)) K(x) \mathrm{d}x$ and
\[
\int^{3\pi/2}_{\pi/2} \tau(x+h)K(x)\mathrm{d}x \geq \int^{3\pi/2}_{\pi/2} z(x)K(x)\mathrm{d}x.
\]
Taking into account (\ref{FReqchoiceepsilon}) and (\ref{FR-defbeta2}), we have $z(\pi/2)>\gamma(\pi/2)$ and $\int^{\pi/2}_{-\pi/2} z(x+\pi)K(x) \mathrm{d}x=\int^{\pi/2}_{-\pi/2} \gamma(x+\pi)K(x) \mathrm{d}x$. 
We can then use Lemma \ref{lemmonotone} to compare the functions $z(x+\pi)$ and $\gamma(x+\pi)$ with respect to the (by Lemma \ref{lemtranslatestest})  non-increasing function $K(x+\pi)/K(x)$ and the positive measure $K(x)\mathrm{d} x$ on the interval $[-\pi/2,\pi/2]$. This yields the inequality
$$
\int^{3\pi/2}_{\pi/2} z(x)K(x)\mathrm{d}x \geq \int^{3\pi/2}_{\pi/2} \gamma(x)K(x)\mathrm{d}x,
$$
establishing (\ref{eqintegralinequalityI}) on $[\pi/2,3\pi/2]$ and thus also on $[-\pi/2,3\pi/2]$.

Let us now show (\ref{eqintegralinequalityI}) for the remaining intervals. We proceed by induction. Suppose we have already shown (\ref{eqintegralinequalityI}) for the intervals $[-\pi/2, L'\pi/2]$, where $L'$ is a positive integer and $L'<L$. Suppose w.l.o.g. that $K(x)$ is negative on $((L-1)\pi/2,L\pi/2)$. (The other case can be treated analogously.) 

First let $L$ be even and set $s:=\tau(h+(L-1)\pi/2) $. If $s \leq \beta_{1}$, the induction hypothesis on $[-\pi/2,(L-1)\pi/2]$, the Lipschitz condition \eqref{Lip1eq}, and the fact that $K$ is negative on $((L-1)\pi/2,L\pi/2)$ imply (\ref{eqintegralinequalityI}) on $J=[-\pi/2, L\pi/2]$. If $s > \beta_{1}$, (\ref{eqintegralinequalityI}) on $[-\pi/2, L\pi/2]$ follows from the Lipschitz condition \eqref{Lip1eq}, the induction hypothesis on $[-\pi/2,(L-2)\pi/2]$, and the fact that $\int^{L\pi/2}_{(L-2)\pi/2}K(x)\mathrm{d}x$ is positive (since $1/(4x^{2}-\pi^{2})$ is decreasing and $K$ is non-negative on $[(L-2)\pi/2, (L-1)\pi/2]$). Indeed,
\begin{align*}
\int^{L\pi/2}_{-\pi/2} \tau(x+h) K(x) \mathrm{d}x & \geq \int^{(L-2)\pi/2}_{-\pi/2} \gamma(x) K(x) \mathrm{d}x + \int^{L\pi/2}_{(L-2)\pi/2} (s + x - (L-1)\pi/2)K(x) \mathrm{d}x\\
& \geq \int^{L\pi/2}_{-\pi/2} \gamma(x) K(x) \mathrm{d}x + (s-\beta_{1})\int^{L\pi/2}_{(L-2)\pi/2}K(x)\mathrm{d}x \\
& \geq \int^{L\pi/2}_{-\pi/2} \gamma(x) K(x) \mathrm{d}x.
\end{align*}

Finally let $L$ be odd. We now set $s:=\tau(h+(L-2)\pi/2)$. We treat the subcase $s < \beta_{1}$ first. We claim that the (\ref{eqintegralinequalityI}) holds on the interval $J=[(L-2)\pi/2,L\pi/2]$. By Lemma \ref{lemintegralfar} (If the upper bound (\ref{eqconditionIs}) is not satisfied, we pick $z$ such that $z(-\pi/2) = s$ and has slope $1$ on $[-\pi/2,\pi/2]$. The proof then only changes mildly.), there is an upper pointed zig-zag function $z(x)$ on the interval $[-\pi/2,\pi/2]$  with $z(-\pi/2)<\beta_1$, 
\begin{equation}
\label{FRcondzigzagintegraleq}
\int^{\pi/2}_{-\pi/2} z(x)K(x) \mathrm{d}x  = \int^{\pi/2}_{-\pi/2} (\beta_{1}+\alpha(x))K(x) \mathrm{d}x ,
\end{equation}
 and
\begin{equation}\label{eqintegralzigzag}
 \int^{L\pi/2}_{(L-2)\pi/2} \tau(x+h) K(x) \mathrm{d}x \geq \int^{\pi/2}_{-\pi/2} z(x) K(x+(L-1)\pi/2) \mathrm{d}x.
\end{equation}
This can be further estimated by 
\begin{equation} \label{eqzigzagjump}
 \int^{\pi/2}_{-\pi/2} z(x) K(x+(L-1)\pi/2) \mathrm{d}x \geq \int^{\pi/2}_{-\pi/2} j(x) K(x+(L-1)\pi/2) \mathrm{d}x,
\end{equation}
where $j$ is the jump function 
\begin{equation}\label{FR-jumpfunction-def}
 j(x) := \begin{cases}
 z(-\pi/2) + (x+\pi/2) & \text{if } x \leq 0\\
 2\beta_{1} -z(-\pi/2)  + \pi/2 -x  & \text{if } x > 0.
\end{cases}
\end{equation} 
Indeed, notice that $\int^{\pi/2}_{-\pi/2} j(x) K(x) \mathrm{d}x = \int^{\pi/2}_{-\pi/2} (\beta_{1}+\alpha(x))K(x) \mathrm{d}x$ and $j(x) = z(x)$ on the interval $[-\pi/2,0]$, because the zig-zag function $z$ attains its peak value on $(0,\pi/2]$ (otherwise \eqref{FRcondzigzagintegraleq} could not be realized due to $z(-\pi/2)<\beta_1$). Hence (\ref{eqzigzagjump}) follows by applying Lemma \ref{lemmonotone} with respect to the (by Lemma \ref{lemtranslatestest}) non-decreasing function $K(x+(L-1)\pi/2)/K(x)$ and the positive measure $K(x)\mathrm{d}x$ on the interval $[0,\pi/2]$. Moreover,
\begin{equation} \label{eqzigzagjump2}
\int^{\pi/2}_{-\pi/2} j(x) K(x+(L-1)\pi/2) \mathrm{d}x \geq \int^{\pi/2}_{-\pi/2} (\beta_{1} + \alpha(x))K(x+(L-1)\pi/2) \mathrm{d}x,
\end{equation}
as follows from $K(-x+(L-1)\pi/2) \leq K(x+(L-1)\pi/2)$ for $0 \leq x \leq \pi/2$. Hence (\ref{eqintegralinequalityI}) holds on $[(L-2)\pi/2,L\pi/2]$ and thus, by applying the induction hypothesis on $[-\pi/2,(L-2)\pi/2 ]$, also on $[-\pi/2, L\pi/2]$.

Now let $s \geq \beta_{1}$. 
By Lemma \ref{lemintegralfar} there is an upper pointed zig-zag function $z(x)$ on $[-\pi/2,\pi/2]$ such that 
\eqref{FRcondzigzagintegraleq}, \eqref{eqintegralzigzag}, and $z(-\pi/2)\leq s$ hold. Notice that the lower bound \eqref{eqconditionIs} has to be satisfied; otherwise the Lipschitz condition \eqref{Lip1eq} and \eqref{eqdefbeta1} would force a contradiction. If $z(-\pi/2)<\beta_1$, we can proceed exactly in the same way as in the previous subcase via the auxiliary jump function \eqref{FR-jumpfunction-def} and show that \eqref{eqzigzagjump} and \eqref{eqzigzagjump2} hold (all we needed there was $z(-\pi/2)<\beta_1$); thus, leading again to (\ref{eqintegralinequalityI}) on $[-\pi/2, L\pi/2]$.
Suppose then that $\beta_{1} \leq z(-\pi/2) \leq s$. Notice that the integral equality \eqref{FRcondzigzagintegraleq}, together with $\beta_{1} \leq z(-\pi/2)$, implies that the peak of $z$ must necessarily occur at some point of the interval $[-\pi/2,0]$; therefore, $z(x-(L-1)\pi/2)) \leq \gamma(x)$ on $[(L-1)\pi/2,L\pi/2]$. We obtain
\begin{align*}
 \int^{L\pi/2}_{-\pi/2}\tau(x+h)K(x)\mathrm{d}x & \geq \int^{(L-2)\pi/2}_{-\pi/2} \tau(x+h)K(x)\mathrm{d}x + \int^{L\pi/2}_{(L-2)\pi/2} z(x-(L-1)\pi/2) K(x)\mathrm{d}x \\
& \geq \int^{(L-3)\pi/2}_{-\pi/2} \gamma(x)K(x)\mathrm{d}x + \int^{(L-2)\pi/2}_{(L-3)\pi/2} (s + (x-(L-2)\pi/2)) K(x)\mathrm{d}x \\
& \quad + \int^{(L-1)\pi/2}_{(L-2)\pi/2} (s + (x-(L-2)\pi/2)) K(x)\mathrm{d}x + \int^{L\pi/2}_{(L-1)\pi/2} \gamma(x)K(x)\mathrm{d}x\\
& \geq \int^{L\pi/2}_{-\pi/2}\gamma(x)K(x)\mathrm{d}x + (s-\beta_{1}) \int^{\pi/2}_{-\pi/2}K(x+(L-2)\pi/2)\mathrm{d}x\\
& \geq \int^{L\pi/2}_{-\pi/2}\gamma(x)K(x)\mathrm{d}x,
\end{align*}
where we have used the induction hypothesis on $[-\pi/2,(L-3)\pi/2]$, the inequality (\ref{eqintegralzigzag}),  the Lipschitz condition \eqref{Lip1eq}, the fact that $K$ is non-negative on $((L-3)\pi/2,(L-2)\pi/2)$, and $\int^{\pi/2}_{-\pi/2}K(x+(L-2)\pi/2)\mathrm{d}x > 0$ (since $1/(4x^{2}-\pi^2)$ is decreasing).

We have shown (\ref{eqintegralinequalityI}) on all required intervals and therefore the proof of (\ref{F-Rproofeq1bis}) is complete.
{
\subsection{Vector-valued functions} It turns out that Theorem \ref{finiteformF-Rth3}  (and hence also Theorem \ref{finiteformF-Rth2}) remains valid for functions with values in a Banach space. As our proof for the scalar-valued version cannot be directly generalized, we discuss here a simple approach to treat the vector-valued case.  We first need a definition. 
\begin{definition} A family $\{G_{\nu}\}_{\nu\in \mathcal{I}}$ of analytic functions on the half-plane $\Re e\:s>0$ is said to have uniform local pseudofunction boundary behavior on the boundary open subset $iU$ if each $G_{\nu}$ has local distributional boundary values there and the boundary distributions
$$
\lim_{\sigma\to 0^{+}}G_{\nu}(\sigma+it)=g_{\nu}(t) \quad (\mbox{in  } \mathcal{D}'(U))
$$
satisfy
$$
\lim_{|x|\to\infty} \widehat{\varphi g_{\nu}}(x)=0 \qquad \mbox{uniformly for } \nu\in \mathcal{I},
$$
for each (fixed) $\varphi\in \mathcal{D}(U)$.

\end{definition}

\smallskip
Our method from Subsection \ref{subsection F-Rproof} also yields the ensuing uniform result.

\begin{lemma}
\label{FRlemma vector-valued}Let $
\{\tau_{\nu}
\}_{\nu\in \mathcal{I}}$ be a family of functions such that $\tau_{\nu}\in\operatorname*{Lip}([0,\infty); M)$ for every $\nu\in \mathcal{I}$. If there is $\lambda>0$ such that the family of Laplace transforms $\mathcal{L}\{\tau_{\nu};s\}$ has uniform local pseudofunction boundary behavior on $(-i\lambda, i\lambda)$, then 
\[
\limsup_{x\to\infty} \sup_{\nu\in \mathcal{I}}|\tau_{\nu}(x)| \leq \frac{M \pi}{2\lambda}.
\]
\end{lemma}

The notion of local pseudofunction boundary behavior immediately extends to analytic functions with values in Banach spaces. We then have,

\begin{theorem}
\label{FRtheorem vector-valued}
Let $E$ be a Banach space and let $\boldsymbol{\tau}:[0,\infty)\to E$ be locally (Bochner) integrable such that, for some sufficiently large  $X>0$,
\begin{equation}
\label{Lip cond eq vector}
\|\boldsymbol{\tau}(x)-\boldsymbol{\tau}(y)\|_{E}\leq M|x-y|, \qquad \mbox{for all } x,y\geq X. 
\end{equation}
If the Laplace transform $\mathcal{L}\{\boldsymbol{\tau};s\}$ has local pseudofunction boundary behavior on $(-i\lambda, i\lambda)$ for some $\lambda>0$, then 
\[
\limsup_{x\to\infty} \|\boldsymbol{\tau}(x)\|_{E} \leq \frac{M \pi}{2\lambda}.
\]
\end{theorem}
\begin{proof} We may assume that \eqref{Lip cond eq vector} holds for all $x,y\in [0,\infty)$. Denote as $E'$ the dual space of $E$.  Applying Lemma \ref{FRlemma vector-valued} to the family of functions 
$$
\tau_{e^{\ast}}(x)=\langle e^{\ast}, \boldsymbol{\tau}(x)\rangle,$$
indexed by $e^{*}$ in the unit ball $B$ of $E'$, we obtain from the Hahn-Banach theorem
$$
\limsup_{x\to\infty} \|\boldsymbol{\tau}(x)\|_{E} = \limsup_{x\to\infty} \sup _{e^{\ast}\in B}|\tau_{e^{\ast}}(x)|\leq \frac{M \pi}{2\lambda}.
$$ 
\end{proof}

\section{Some generalizations} \label{section generalizations and corollaries F-R}
We now discuss some generalizations and consequences of Theorem \ref{finiteformF-Rth3}. We start with a general inequality for functions whose Laplace transforms have local pseudofunction behavior on a given symmetric segment of the imaginary axis.

Given a function $\tau$ and a number $\delta>0$, define its oscillation modulus (at infinity) as the non-decreasing function
\[
\Psi(\delta):=\Psi(\tau,\delta)= \limsup_{x\to\infty}\sup_{h\in [0,\delta]}|\tau(x+h)-\tau(x)|.
\]
The oscillation modulus is involved in the definition of many familiar and important Tauberian conditions. For example, the function $\tau$ is boundedly oscillating precisely when $\Psi$ is finite for some $\delta$, while it is slowly oscillating if $\Psi(0^{+})=\lim_{\delta\to0^{+}}\Psi(\delta)=0$. Since $\Psi$ is subadditive, it is finite everywhere whenever $\tau$ is boundedly oscillating. We shall call a function R-slowly oscillating (regularly slowly oscillating) if $\limsup_{\delta\to0^{+}}\Psi(\delta)/\delta<\infty$. Since $\Psi$ is subadditive, it is easy to see the latter implies that $\Psi$ is right differentiable at $\delta=0$ and indeed 
$$
\Psi'(0^{+})=\sup_{\delta>0}\frac{\Psi(\delta)}{\delta}.
$$
Observe that if $\tau\in\operatorname*{Lip}([X,\infty), M)$, then $\Psi'(0^{+})\leq M$.

\begin{theorem}
\label{finiteformF-Rth4} Let $\tau\in L^{1}_{loc}[0,\infty)$ be such that 
\begin{equation}
\label{FR-eqLapTconvibis}
 \mathcal{L}\{\tau;s\} = \int^{\infty}_{0} \tau(x) e^{-sx} \mathrm{d}x \text{ converges for } \Re e \: s > 0
\end{equation}
and admits local pseudofunction boundary behavior on the segment $(-i\lambda, i\lambda)$. Then,
\begin{equation}
\label{F-Rgceq1}
\limsup_{x\to\infty}|\tau(x)|\leq \inf_{\delta>0} \left(1+\ \frac{\pi}{2\delta\lambda} \right) \Psi(\delta).
\end{equation}
Furthermore, if $\tau$ is R-slowly oscillating, then
\begin{equation}
\label{F-Rgceq2}
\limsup_{x\to\infty}|\tau(x)|\leq \frac{\pi \Psi'(0^{+})}{2\lambda}.
\end{equation}
\end{theorem} 
\begin{proof} We can of course assume that $\tau$ is boundedly oscillating; otherwise the right side of \eqref{F-Rgceq1} is identically infinity and the inequality trivially holds. We follow an idea of Ingham \cite{ingham1935} and reduce our problem to an application of Theorem \ref{finiteformF-Rth3}. Fix $\delta$ and let $M>\Psi(\delta)$ be arbitrary but also fixed. There is $X>0$ such that
$$
|\tau(x)-\tau(y)|< M, \quad \mbox{for all } x\geq X \mbox{ and } x\leq y \leq x+\delta.
$$
Define 
\begin{equation}
\label{FRtaudelta}
\tau_{\delta}(x)=\frac{1}{\delta}\int_{x}^{x+\delta} \tau(u)\mathrm{d}u. 
\end{equation}
for $x\geq X$ and $\tau_{\delta}(x)=0$ otherwise. Then, 
$$
|\tau_{\delta}(x+h)-\tau_{\delta}(x)|\leq \frac{1}{\delta}\int_{x}^{x+h}|\tau(u+\delta)-\tau(u))|\mathrm{d}u\leq \frac{M}{\delta} h, \quad x\geq X, \: h\geq 0,
$$
that is, $\tau_{\delta}\in \operatorname*{Lip}([X,\infty); M/\delta)$. Its Laplace transform is given by
$$
\mathcal{L}\{\tau_{\delta};s\}= \frac{e^{\delta s}-1}{\delta s}\mathcal{L}\{\tau;s\} + \text{entire function}, \quad \Re e\: s>0,
$$
and hence also has local pseudofunction boundary behavior on $(-i\lambda, i\lambda)$.  Theorem \ref{finiteformF-Rth3} implies that
$$
\limsup_{x\to\infty} \frac{1}{\delta}\left|\int_{x}^{x+\delta} \tau(u)\mathrm{d}u\right|\leq \frac{M\pi}{2\delta\lambda}.
$$
Therefore,
$$
\limsup_{x\to\infty}|\tau(x)|\leq \frac{M\pi}{2\delta\lambda}+\limsup_{x\to\infty} \frac{1}{\delta}\int_{x}^{x+\delta} |\tau(u)-\tau(x)|\mathrm{d}u \leq \left(1+\ \frac{\pi}{2\delta\lambda} \right) M,
$$
whence \eqref{F-Rgceq1} follows. 

Assume now that $\tau$ is R-slowly oscillating. Then,
$$
\lim\sup_{x\to\infty}|\tau(x)|\leq  \lim_{\delta\to0^{+}}\left(1+\ \frac{\pi}{2\delta\lambda} \right) \Psi(\delta)=\frac{\pi \Psi'(0^{+})}{2\lambda}.
$$

\end{proof}

It should be noticed that the inequalities \eqref{F-Rgceq1} and \eqref{F-Rgceq2} are basically sharp in the following sense. If we take as $\tau$ the example from Subsection \ref{F-R subsection example}, one has for this function $\Psi'(0)=1$ and
$$
\frac{\pi}{2}=\limsup_{x\to\infty}|\tau(x)|=\inf_{\delta>0} \left(1+\ \frac{\pi}{2\delta} \right) \Psi(\delta),
$$
which shows that the constant $\pi/2$ in \eqref{F-Rgceq1} and \eqref{F-Rgceq2} is optimal.

The next result improves another Tauberian theorem of Ingham\footnote{Ingham's result is $\limsup_{x\to\infty}|\tau(x)|\leq 2(1+3\theta/\lambda)\Theta$.} (cf. \cite[Thm.~I, p.~464]{ingham1935}). It plays an important role for our treatment of one-sided Tauberian conditions in the following section.

\begin{theorem}\label{finiteformF-Rth5}
Let $\tau$ be of local bounded variation, vanish on $(-\infty,0)$, have convergent Laplace transform \eqref{FR-eqLapTconvibis}, and satisfy the Tauberian condition
\begin{equation}
\label{FR-gceq3}
\limsup_{x\to\infty}e^{-\theta x}\left|\int_{0^{-}}^{x}e^{\theta u}\mathrm{d}\tau(u)\right|=:\Theta<\infty,
\end{equation}
for some $\theta>0$. If $\mathcal{L}\{\tau;s\}$ has local pseudofunction boundary behavior on $(-i\lambda, i\lambda)$, then
\begin{equation}
\label{FR-gceq4}
\limsup_{x\to\infty}\left|\tau(x)\right| \leq \left(1+\frac{\theta\pi}{2\lambda}\right)\Theta.
\end{equation}
\end{theorem}
\begin{proof} We apply our method from the proof of \cite[Cor. 5.11]{Debruyne-VindasComplexTauberians}, but taking into account the sharp value $\pi/2$ in Theorem \ref{finiteformF-Rth2}. Define 
$$\rho(x)= e^{-\theta x}\int_{0^{-}}^{x}e^{\theta u}\mathrm{d}\tau(u).$$
Integrating by parts,
\begin{equation}
 \label{FR-gceq5}
 \tau(x)=\rho(x)+\theta\int_{0}^{x}\rho(u)\mathrm{d}u.
 \end{equation}
The relation (\ref{FR-gceq3}) is the same as $\limsup_{x\to\infty}|\rho(x)|\leq \Theta$. Thus, using Theorem \ref{finiteformF-Rth2} and \eqref{FR-gceq5}, the inequality \eqref{FR-gceq4} would follow if we verify that $\mathcal{L}\{\rho;s\}/s$ has local pseudofunction boundary behavior on $(-i\lambda,i\lambda)$. For it, notice we have that 
\begin{equation*}
\frac{\mathcal{L}\{\rho;s\}}{s}=\frac{\mathcal{L}\{\tau;s\}}{s+\theta}. 
\end{equation*}
The function $1/(s+\theta)$ is $C^{\infty}$ on $\Re e\: s=0$, and thus a multiplier for local pseudofunctions. This shows that $\mathcal{L}\{\rho;s\}/s$ has local pseudofunction boundary behavior on $(-i\lambda,i\lambda)$, as required. 
\end{proof}
\begin{remark}
The constants $\pi/2$ and $1$ in Theorem \ref{finiteformF-Rth5} are also optimal in the sense that if
\begin{equation}
\label{eqoptInghamrefined}
\limsup_{x\to\infty}\left|\tau(x)\right| \leq \left(\mathfrak{L} +\frac{\theta \mathfrak{M}}{\lambda}\right)\Theta,
\end{equation}
holds for all $\theta$ and all functions satisfying the hypotheses of the theorem, then $\mathfrak{M}\geq \pi/2$ and $\mathfrak{L} \geq 1$. To see this, take as $\tau$ again the example from Subsection \ref{F-R subsection example}.  (As usual, we normalize the situation with $\lambda=1$.) For this function, we have
$$
\Theta=\limsup_{x\to\infty}e^{-\theta x}\left|\int_{0}^{x}e^{\theta u}\tau'(u)\mathrm{d}u\right|= \frac{e^{\pi \theta} - 1}{\theta(1+e^{\pi \theta})}\: .
$$
Inserting this in \eqref{eqoptInghamrefined}, we obtain
$$\frac{\pi}{2}\leq \left(\frac{\mathfrak{L}}{\theta}+ \mathfrak{M}\right)\frac{e^{\pi \theta} - 1}{1+e^{\pi\theta}},$$
which gives $\mathfrak{M}\geq \pi/2$ after taking $\theta\to\infty$ and $\mathfrak{L} \geq 1$ after taking $\theta \to 0^{+}$.
\end{remark}

We end this section with an improved version of Theorem \ref{finiteformF-Rth2} where one allows a closed null boundary subset of possible singularities for the Laplace transform. We remark that Theorem \ref{finiteformF-Rth6} improves a theorem of Arendt and Batty from \cite{a-b} and that these kinds of Tauberian results have been extensively applied in the study of asymptotics of $C_0$-semigroups; see \cite[Chap. 4]{a-b-h-n} for an overview of results, especially when the singular set $E$ is countable and one has the stronger hypothesis of analytic continuation. Theorem \ref{finiteformF-Rth6} follows immediately by combining Theorem \ref{finiteformF-Rth2} and our characterization of local pseudofunctions \cite[Thm. 4.1]{Debruyne-VindasComplexTauberians}.

\begin{theorem}\label{finiteformF-Rth6} 
Let $\rho\in L^{1}_{loc}[0,\infty)$ have convergent Laplace transform on $\Re e\: s>0$. 
Suppose that there are $\lambda>0$, a closed null set $0\not\in E\subset\mathbb{R}$ such that  
\begin{equation}
\label{ttheq17}
\sup_{x>0}\left| \int_{0}^{x}\rho(u)e^{-itu}\mathrm{d}u\right|<M_{t}<\infty \quad \mbox{for each } t\in  E\cap(-\lambda,\lambda),
\end{equation}
and a constant $b$ such that
\begin{equation}
\label{ttheeq18}
 \frac{\mathcal{L}\{\rho; s\}-b}{s}
\end{equation}
has local pseudofunction boundary behavior on $(-i\lambda,i\lambda)\setminus iE$. Then, 
\[
\label{ttheq23} \limsup_{x\to\infty}\left|\int_{0}^{x}\rho(u)\mathrm{d}u-b\right|\leq\frac{\pi}{2\lambda}\limsup_{x\to\infty}|\rho(x)| .
\]
\end{theorem}
\begin{remark} What we have shown in \cite[Thm. 4.1]{Debruyne-VindasComplexTauberians} is that if \eqref{ttheq17} holds on the closed null exceptional set, then actually \eqref{ttheeq18} has local pseudofunction boundary behavior on the whole segment $(-i\lambda,i\lambda)$. This consideration becomes very meaningful when one works with stronger boundary conditions. For example, if (\ref{ttheeq18}) is regular at every point of $(-i\lambda,i\lambda)\setminus{iE}$ and \eqref{ttheq17} is satisfied, then $iE$ may still be a singular set for analytic continuation, though $iE$ becomes no longer singular for local pseudofunction boundary behavior.
\end{remark}

\begin{remark} It is important to notice that, in view of Theorem \ref{FRtheorem vector-valued}, all results from this section admit immediate generalizations for functions with values in Banach spaces. We leave the formulation of such vector-valued versions to the reader. For Theorem \ref{finiteformF-Rth6},    note that the proof of \cite[Thm. 4.1]{Debruyne-VindasComplexTauberians} also applies to obtain a corresponding characterization of Banach space valued local pseudofunctions. 
\end{remark}
\section{One-sided Tauberian hypotheses}\label{section one-sided F-R}

We study in this section Ingham type Tauberian theorems with one-sided Tauberian conditions. We begin with a one-sided version of Theorem \ref{finiteformF-Rth3}. 

\begin{theorem} \label{finiteformF-R1sidedth1}
Let $\tau\in L^{1}_{loc}[0,\infty)$. Suppose there are constants $M,X>0$ such that $\tau(x)+Mx$ is non-decreasing on $[X,\infty)$.
If 
\begin{equation}
\label{FR-eqLapTconv}
 \mathcal{L}\{\tau;s\} = \int^{\infty}_{0} \tau(x) e^{-sx} \mathrm{d}x \text{ converges for } \Re e \: s > 0
\end{equation}
and admits local pseudofunction boundary behavior on the segment $(-i\lambda, i\lambda)$, then
\begin{equation}
\label{FR-inequality-pi}
 \limsup_{x\rightarrow \infty} \left|\tau(x)\right| \leq \frac{M\pi}{\lambda}.
\end{equation}
The constant $\pi$ in \eqref{FR-inequality-pi} is best possible.
\end{theorem}
\begin{proof} We combine Ingham's idea from \cite[pp. 472--473]{ingham1935} with the Graham-Vaaler sharp version of the Wiener-Ikehara theorem \cite{grahamvaaler} and Theorem \ref{finiteformF-Rth5}. Let $\theta>0$. We may assume that $X=0$. Also notice that since $\tau(x) + Mx$ is non-decreasing, $\tau$ is of local bounded variation. Our Tauberian assumption on $\tau$ is that $\mathrm{d}\tau(x)+M\mathrm{d}x$ is a positive measure on $[0,\infty)$. Consider the non-decreasing function 
$S(x)=\int_{0^{-}}^{x}e^{u\theta}(\mathrm{d}\tau(u)+M\mathrm{d}u)$.
Its Laplace-Stieltjes transform is
$$
\mathcal{L}\{\mathrm{d}S;s\}= (s-\theta)\mathcal{L}\{\tau;s-\theta\}+\frac{M}{s-\theta}, \quad \Re e\:s>\theta,
$$
and hence 
$$
\mathcal{L}\{\mathrm{d}S;s\}-\frac{M}{s-\theta}
$$
has local pseudofunction pseudofunction boundary behavior on $(\theta-i\lambda,\theta + i\lambda)$. The sharp Wiener-Ikehara theorem\footnote{Graham and Vaaler state their theorem with the boundary hypothesis of continuous extension \cite[Thm. 10, p. 294]{grahamvaaler}, while Korevaar works with $L^{1}_{loc}$-boundary behavior \cite[Thm. III.5.4, p. 130]{korevaarbook}. Since pseudomeasure boundary behavior of $\mathcal{L}\{\mathrm{d}S;s\}$ at $s=\theta$ yields \cite[Prop. 3.1]{Debruyne-VindasW-I}  $e^{-\theta x}S(x)=O(1)$, a small adaptation via a density argument in the proof given in \cite[p. 130--131]{korevaarbook} shows that the Graham-Vaaler theorem remains valid under the more general hypothesis of local pseudofunction boundary behavior.} yields
$$
\frac{2\pi\theta/\lambda}{e^{2\pi\theta/\lambda}-1} \: \frac{M}{\theta}
 \leq \liminf_{x\to\infty}e^{-\theta x}S(x)\leq \limsup_{x\to\infty}e^{-\theta x}S(x)\leq \frac{2\pi\theta/\lambda}{1-e^{-2\pi\theta/\lambda}} \: \frac{M}{\theta}.
$$
Consequently,
\begin{align*}
&\limsup_{x\to\infty}e^{-\theta x}\left|\int_{0^{-}}^{x}e^{u\theta}\mathrm{d}\tau(u)\right|
\\
&
\quad \quad \quad \quad
\leq \frac{M}{\theta (e^{2\pi\theta/\lambda}-1)} \max\{e^{2\pi\theta/\lambda}-1-2\pi\theta/\lambda, e^{2\pi\theta/\lambda}2\pi\theta/\lambda- e^{2\pi\theta/\lambda}+1  \}
\\
&
\quad \quad\quad \quad
=
\frac{M(e^{2\pi\theta/\lambda}2\pi\theta/\lambda- e^{2\pi\theta/\lambda}+1)}{\theta(e^{2\pi\theta/\lambda}-1)} .
\end{align*}
Applying Theorem \ref{finiteformF-Rth5} and setting $u=2\pi \theta/\lambda$, we obtain
$$
\limsup_{x\to\infty}|\tau(x)|\leq \frac{\pi M}{\lambda} \left(1+\frac{u}{4}\right)\frac{2(ue^{u}-e^u+1)}{u(e^{u}-1)}.
$$
This inequality is valid for all $u>0$. Taking the limit as $u\to0^{+}$, we obtain \eqref{FR-inequality-pi}. The optimality of the constant $\pi$ is shown in Example \ref{FR-ex1} below.
\end{proof}

Theorem \ref{finiteformF-R1sideth2} is an immediate corollary of Theorem \ref{finiteformF-R1sidedth1} (except for the sharpness of $\pi$ there that is checked below). The next generalization of Theorem \ref{finiteformF-R1sidedth1} can be shown via the simple reduction used in the Theorem \ref{finiteformF-Rth4}. Define the decrease modulus (at infinity) of a function $\tau$ as the non-decreasing subadditive function
\[
\Psi_{-}(\delta):=\Psi_{-}(\tau,\delta)= -\liminf_{x\to\infty}\inf_{h\in [0,\delta]}\tau(x+h)-\tau(x), \quad \delta>0.
\]
Notice $\Psi_{-}$ is non-negative. A function $\tau$ is said to be boundedly decreasing if $\Psi_{-}(\delta)$ is finite for some (and hence all) $\delta>0$. It is slowly decreasing if $\Psi_{-}(0^{+})=0$, and we call $\tau$  R-slowly decreasing (regularly slowly decreasing) if 
$$
\Psi'_{-}(0^{+})=\sup_{\delta>0}\frac{\Psi_{-}(\delta)}{\delta}<\infty.
$$

\begin{theorem}
\label{finiteformF-R1sidedth3} Let $\tau\in L^{1}_{loc}[0,\infty)$ be such that \eqref{FR-eqLapTconv} holds. If $\mathcal{L}\{\tau;s\}$
admits local pseudofunction boundary behavior on the segment $(-i\lambda, i\lambda)$, then
\begin{equation}
\label{FR1sidedeq1}
\limsup_{x\to\infty}|\tau(x)|\leq \inf_{\delta>0} \left(1+\ \frac{\pi}{\delta\lambda} \right) \Psi_{-}(\delta).
\end{equation}
Furthermore, if $\tau$ is R-slowly decreasing, then
\begin{equation}
\label{FR1sidedeq2}
\limsup_{x\to\infty}|\tau(x)|\leq \frac{\pi \Psi'_{-}(0^{+})}{\lambda}.
\end{equation}
\end{theorem} 
\begin{proof} Fix $\delta$ and let $M>\Psi_{-}(\delta)$. The function
(\ref{FRtaudelta}) satisfies that $\tau_{\delta}(x)+Mx/\delta$ is non-decreasing for all $x\geq X$, when $X>0$ is sufficiently large. Applying Theorem \ref{finiteformF-R1sidedth1} to it, we get
$$
\limsup_{x\to\infty} \frac{1}{\delta}\left|\int_{x}^{x+\delta} \tau(u)\mathrm{d}u\right|\leq \frac{M\pi}{\delta\lambda}.
$$
Now, 
$ \tau(x)\leq \tau_{\delta}(x)+M$ and $\tau_{\delta}(x)-M\leq \tau(x + \delta)$ for $x\geq X$, whence we obtain \eqref{FR1sidedeq1}. The inequality \eqref{FR1sidedeq2} follows from \eqref{FR1sidedeq1} if $\tau$ is additionally R-slowly decreasing.
\end{proof}
Let us now  give two examples for the optimality of  Theorem \ref{finiteformF-R1sidedth1} and Theorem \ref{finiteformF-R1sideth2}.

\begin{example} \label{FR-ex1} Let
\begin{equation}\label{FR1sidedeq3}
 \tau(x) = \begin{cases} 
0 & \mbox{if }  x \leq 0, \\
-x & \mbox{if } 0 \leq x \leq 1,\\
-x + N & \mbox{if } N-1 < x \leq N+1 \mbox{ for even } N.
\end{cases}
\end{equation}
Calculating its Laplace transform, one gets
\begin{equation}
\label{FR1sidedeq4}
 \mathcal{L}\{\tau;s\} = -\frac{1}{s^2} + \frac{2e^{-s}}{s(1-e^{-2s})}, \ \ \ \Re e \: s > 0,
\end{equation}
which admits analytic extension to $(-i\pi , i\pi )$, and thus also has local pseudofunction boundary behavior on that boundary segment. Since $M = 1$, $\lambda = \pi$, and $\limsup_{x\to\infty} \left|\tau(x)\right| = 1$, the inequality (\ref{FR-inequality-pi}) cannot hold with a better value than $\pi$. An appropriate transformation of this example will then show the sharpness for arbitrary $M$ and $\lambda$. 
\end{example}

\begin{example} \label{FR-ex2} To show the sharpness of \eqref{F-Rintroeq3} in Theorem \ref{finiteformF-R1sideth2}, it suffices to construct a sequence of bounded functions such that $\operatorname*{supp}\rho_{n}\subseteq[0,\infty)$, $\liminf_{x\to\infty}\rho_n(x)=-1$, their Laplace transforms $\mathcal{L}\{\rho_n;s\}$ have analytic extension to $(-i\pi,i\pi)$ with $\mathcal{L}\{\rho_n;0\}=0$, and 
\begin{equation}
\label{FR1sidedeq5}
\lim_{n\to\infty} \limsup_{x\to\infty}\left|\int_{0}^{x}\rho_{n}(u)\mathrm{d}u\right|=1.
\end{equation}
We consider smooth versions of the (distributional) derivative of $\tau$ given by \eqref{FR1sidedeq3}. Let $\psi\in\mathcal{S}(\mathbb{R})$ be a non-negative test function such that $\operatorname*{supp}\psi\subseteq(1,3)$ and $\int_{-\infty}^{\infty}\psi(x)\mathrm{d}x=1$. Set $\psi_{n}(x)=n\psi(nt)$ and $$\rho_{n}(x)=(\psi_{n} \ast \mathrm{d}\tau)(x)=-\int_{0}^{x}\psi_n(u)\mathrm{d}u+2\sum_{k=0}^{\infty}\psi_{n}(x-2k-1).$$
The smooth functions $\rho_n$ are all supported in $[0,\infty)$ and clearly 
$$\liminf_{x\to\infty}\rho_n(x)=\min_{x\in\mathbb{R}}\rho_n(x)=-1.$$
Furthermore, using (\ref{FR1sidedeq4}), their Laplace transforms $\mathcal{L}\{\rho;s\}$ extend to $(-i\pi , i\pi )$ analytically as
$$
\mathcal{L}\{\rho_n;it\}=\hat{\psi}(t/n)\left( -\frac{1}{it} + \frac{2e^{-it}}{1-e^{-2it}}\right), \quad t\in (-\pi,\pi),
$$  
and $\mathcal{L}\{\rho_n;0\}=\hat{\psi}(0)\cdot 0=0$.
Also, 
$$
\int_{0}^{x}\rho_n(u)\mathrm{d}u=(\psi_{n}\ast \tau)(x).
$$
Since $\tau$ is uniformly continuous on any closed set $\mathbb{R}\setminus(\bigcup_{n\in\mathbb{N}} (2n+1-\varepsilon,2n+1+\varepsilon)  )$, we have that $\psi_n\ast \tau$ converges uniformly to $\tau$ on any closed set $\mathbb{R}\setminus(\bigcup_{n\in\mathbb{N}} (2n+1-\varepsilon,2n+1+\varepsilon)  )$. Therefore, (\ref{FR1sidedeq5}) holds. 
\end{example}
\begin{remark}
We can also use our convolution method from \cite{Debruyne-VindasComplexTauberians} to get a value for the constant in Theorem \ref{finiteformF-R1sidedth1}. Although the optimal constant $\pi$ seems then to be out of reach, that simple method delivers a much better constant than Ingham's (cf. \eqref{F-Rintroeq4}). For example, we discuss here how to obtain the weak inequality
$$
\limsup_{x\to\infty}|\tau(x)|\leq \frac{4.1\: M}{\lambda}
$$
under the assumptions of Theorem \ref{finiteformF-R1sidedth1}. As was the case for the two-sided Tauberian condition, we may suppose that $M = \lambda = 1$ by an appropriate transformation. The Tauberian condition implies that $\tau$ is boundedly decreasing (in the sense defined in \cite[Sect. 3]{Debruyne-VindasComplexTauberians}). Hence, we deduce from \cite[Thm.  3.1]{Debruyne-VindasComplexTauberians} that $\tau$ is bounded near $\infty$. We may then suppose without loss of generality that $\tau\in L^{\infty}(\mathbb{R})$. We may also assume that the Tauberian condition holds globally, that is,
\begin{equation}
\label{FR-Taubcond1sided1}
\tau(x+h)-\tau(x)\geq -h, \quad \mbox{for all }x\in\mathbb{R} \mbox{ and } h\geq0.
\end{equation}
 We let $S := \limsup_{x\to\infty} \left|\tau(x)\right|$. As in the proof of Theorem \ref{finiteformF-Rth3} (see Subsection \ref{subsection F-Rproof}),  the local pseudofunction boundary behavior of the Laplace transform translates into
\[
\int^{\infty}_{-\infty} \tau(x+h)\phi(x)\mathrm{d}x = o(1),
\]
for all $\phi \in L^{1}(\mathbb{R})$ whose Fourier transform vanishes outside the interval $[-1.1]$. We pick the F\'{e}jer kernel
\[
 \phi(x) = \left(\frac{\sin{(x/2)}}{x/2}\right)^{2}.
\]
Suppose that $S > 4.1$. Let $\varepsilon > 0$ be a sufficiently small constant; more precisely, we choose it such that
$$
9.79 \approx 2\int^{5.85}_{-2.35}\phi(x)\mathrm{d}x > (1+\varepsilon/(S-4.1)) \int^{\infty}_{-\infty}\phi(x)\mathrm{d}x = 2\pi (1 + \varepsilon/(S-4.1))
$$
and
\[
25.77\approx \int^{5.85}_{-2.35} (8.2 - (x+2.35))\phi(x) \mathrm{d}x > 4.1 \int^{\infty}_{-\infty}\phi(x)\mathrm{d}x  + \varepsilon \approx 25.76 + \varepsilon.
\]
Then there exists $Y$ such that $\int^{\infty}_{-\infty} \tau(x+Y+2.35)\phi(x)\mathrm{d}x \leq \varepsilon$ and $\tau(Y) \geq S - \varepsilon$. (The case $\tau(Y) \leq -4.1 + \varepsilon$ can be treated similarly.) We may additionally assume that $\tau(x) \geq -S -\varepsilon$ for all $x$. (Here we note that the $\varepsilon$ that gives the contradiction does not depend on $\tau$, but only on $S$ and some other absolute constants.) Since $\phi$ is nonnegative and $\tau$ satisfies \eqref{FR-Taubcond1sided1}, it follows that
\begin{align*}
\int^{\infty}_{-\infty} \tau(x+Y+2.35)\phi(x) \mathrm{d}x & \geq \int^{-2.35}_{-\infty} (-S-\varepsilon)\phi(x) \mathrm{d}x + \int^{5.85}_{-2.35}(S-\varepsilon -x - 2.35)\phi(x) \mathrm{d}x \\
& \quad + \int^{\infty}_{5.85}(-S-\varepsilon)\phi(x)\mathrm{d}x \\
& \geq \int^{-2.35}_{-\infty} -4.1\phi(x) \mathrm{d}x + \int^{5.85}_{-2.35}(4.1-x - 2.35)\phi(x) \mathrm{d}x \\
& \quad + \int^{\infty}_{5.85}-4.1\phi(x)\mathrm{d}x \\
& = \int^{5.85}_{-2.35} (8.2- x - 2.35)\phi(x) \mathrm{d}x - \int^{\infty}_{-\infty} 4.1 \phi(x) \mathrm{d}x > \varepsilon,
\end{align*}
establishing a contradiction. Therefore, we must have $S \leq 4.1$.
\end{remark}


\begin{thebibliography}{99}  


\bibitem{a-b} W.~Arendt, C.~J.~K.~Batty, \emph{Tauberian theorems and stability of one-parameter semigroups,} Trans. Amer. Math. Soc. \textbf{306} (1988), 837--852. 

\bibitem{a-b-h-n}W.~Arendt, C.~J.~K.~Batty, M.~Hieber, F.~Neubrander, \emph{Vector-valued Laplace transforms and Cauchy problems,} Second edition, Birkh\"{a}user/Springer Basel AG, Basel, 2011.

\bibitem{b-b-t2016} C.~J.~K.~Batty, A.~Borichev, Y.~Tomilov, \emph{$L^{p}$-tauberian theorems and $L^{p}$-rates for energy decay,} J. Funct. Anal. \textbf{270} (2016), 1153--1201.

\bibitem {Chill1998} R.~Chill, \emph{Tauberian theorems for vector-valued Fourier and Laplace transforms,} Studia Math. \textbf{128} (1998), 55--69. 

\bibitem{Chill2016} R. Chill, D.~Seifert, \emph{Quantified versions of Ingham's theorem,} Bull. Lond. Math. Soc. \textbf{48} (2016), 519--532. 

\bibitem{Debruyne-VindasW-I} G.~Debruyne, J.~Vindas, \emph{Generalization of the Wiener-Ikehara theorem,} Illinois J. Math. \textbf{60} (2016), 613--624.

\bibitem{Debruyne-VindasComplexTauberians} G.~Debruyne, J.~Vindas, \emph{Complex Tauberian theorems for Laplace transforms with local pseudofunction boundary behavior,} J. Anal. Math., to appear (preprint arXiv:1604.05069).

\bibitem{grahamvaaler} S.~W.~Graham, J.~D.~Vaaler, \emph{A class of extremal functions for the Fourier transform,} Trans. Amer. Math. Soc. \textbf{265} (1981), 283--302.



\bibitem{ingham1935} A.~E.~Ingham, \emph{On Wiener's method in Tauberian theorems,} Proc. London Math. Soc. (2) \textbf{38} (1935), 458--480. 

\bibitem{ingham1936} A.~E.~Ingham, \emph{Some trigonometrical inequalities with applications to the theory of series,} Math. Z. \textbf{41} (1936), 367--379. 


\bibitem{Kolk2003} J.~A.~C.~ Kolk, \emph{On Euler numbers, Hilbert sums, Lobachevski\u{\i} integrals, and their asymptotics,} Indag. Math. (N.S.) \textbf{14} (2003), 445--449. 

\bibitem{korevaar1982} J.~Korevaar, \emph{On Newman's quick way to the prime number theorem,} Math. Intelligencer \textbf{4} (1982), 108--115.

\bibitem{korevaarbook} J.~Korevaar, \emph{Tauberian theory. A century of developments}, Grundlehren der Mathematischen Wissenschaften, 329, Springer-Verlag, Berlin, 2004.

\bibitem{korevaar2005} J.~Korevaar, \emph{Distributional Wiener-Ikehara theorem and twin primes,} Indag. Math. (N.S.) \textbf{16} (2005), 37--49.

\bibitem{korevaar2005FR} J.~Korevaar, \textit{A Tauberian theorem for Laplace transforms with pseudofunction boundary behavior,} in: Complex analysis and dynamical systems II, pp. 233--242, Contemp. Math., 382, Amer. Math. Soc., Providence, RI, 2005.

\bibitem{logan1983} B.~F.~Logan, \emph{Extremal problems for positive-definite bandlimited functions. II. Eventually negative functions,} SIAM J. Math. Anal. \textbf{14} (1983), 253--257. 



\bibitem{zagier1997} D.~Zagier, \emph{Newman's short proof of the prime number theorem,} Amer. Math. Monthly \textbf{104} (1997), 705--708. 


\end{thebibliography}
\end{document}